\documentclass[12pt]{amsart}
\usepackage{amsmath}
\usepackage{amssymb}
\usepackage{hyperref}
\usepackage{epsfig}
\usepackage{graphicx}
\numberwithin{equation}{section}
\usepackage{color}

\input xy
\xyoption{all}

\calclayout
\allowdisplaybreaks[3]

\theoremstyle{plain}
\newtheorem{prop}{Proposition}

\newtheorem{theo}[prop]{Theorem}
\newtheorem{coro}[prop]{Corollary}

\theoremstyle{definition}

\newtheorem{rema}[prop]{Remark}

\def\ra{\rightarrow}

\def\cH{{\mathcal H}}

\def\cO{{\mathcal O}}

\def\cX{{\mathcal X}}

\def\bP{{\mathbb P}}

\def\bZ{{\mathbb Z}}
\def\bR{{\mathbb R}}

\def\bC{{\mathbb C}}

\def\bF{{\mathbb F}}

\def\Br{\mathrm{Br}}
\def\CH{\mathrm{CH}}

\def\Gr{\mathrm{Gr}}

\def\lim{\mathrm{lim}}

\makeatother
\makeatletter

\author{Brendan Hassett}
\address{Department of Mathematics\\
Brown University \\
Box 1917 
151 Thayer Street
Providence, RI 02912 \\
USA}
\email{bhassett@math.brown.edu}

\author{Alena Pirutka}
\address{Courant Institute\\
                New York University \\
                New York, NY 10012 \\
                USA }
\email{pirutka@cims.nyu.edu}

\author{Yuri Tschinkel}
\address{Courant Institute\\
                New York University \\
                New York, NY 10012 \\
                USA }
\email{tschinkel@cims.nyu.edu}

\address{Simons Foundation\\
160 Fifth Avenue\\
New York, NY 10010\\
USA}

\title[Intersections of three quadrics]{Intersections of three quadrics in $\bP^7$}

\begin{document}
\date{\today}

\begin{abstract}
We study rationality properties of smooth complete intersections of three quadrics in $\bP^7$. 
We exhibit a smooth family of such intersections
with both rational and non-rational fibers.
\end{abstract}

\maketitle

\section{Introduction}

The specialization method, introduced by Voisin \cite{Voisin}, 
and developed by Colliot-Th\'el\`ene--Pirutka, Totaro, and others, has led to major advances in higher-dimensional complex birational 
geometry. It makes it possible, for the first time, to prove failure of stable rationality of some smooth quartic threefolds \cite{ct-pirutka}, 
cyclic covers \cite{Voisin}, \cite{beau-6}, \cite{ct-pir-cyclic}, \cite{okada}, and
large degree smooth Fano hypersurfaces in projective space \cite{totaro-JAMS}.

The specialization method yields failure of stable rationality of a very general member of a family of complex
algebraic varieties from the existence of a single, mildly singular, fiber with an explicit obstruction, 
that can be formulated in terms of integral decomposition of the diagonal or universal $\CH_0$-triviality 
(see Section~\ref{sect:stab-rat} for more details and references). 
A surprising aspect of applications of the method was that {\em a priori} different families 
of varieties admit specializations to the same `reference varieties'.
This allows us to propagate the failure of stable rationality, by finding 
suitable chains of specializations. Examples of such `reference varieties' are conic or quadric surface bundles over rational
surfaces, with carefully chosen discriminant loci (see \cite{pirutka-survol}).     
A similar approach -- via specialization to quartic del Pezzo fibrations
over $\bP^1$ --
may be used to essentially settle the stable rationality problem 
for very general smooth rationally connected threefolds
\cite{HKT-conic}, \cite{HT-fano}, \cite{krylov}, with the exception of cubic threefolds, 
whose stable rationality remains elusive \cite{VoisinJEMS}. 

New effects arise in dimension four: 
rationality properties can change in smooth families \cite{HPT}. The relevant reference variety is
$Y\subset \bP^2_{\lambda}\times \bP^3_y$, given by the vanishing of the $(2,2)$ form   
\begin{equation}
\label{eqn:Y}
\lambda_1\lambda_2y_0^2+ \lambda_0\lambda_2y_1^2+ \lambda_0\lambda_1y_2^2+F(\lambda_0,\lambda_1,\lambda_2) y_3^2,
\end{equation}
with 
\begin{equation}
\label{eqn:F}
F(\lambda_0,\lambda_1,\lambda_2):=\lambda_0^2+\lambda_1^2+\lambda_2^2-2(\lambda_0\lambda_1+\lambda_0\lambda_2+\lambda_1\lambda_2)
\end{equation}
defining a conic tangent to each coordinate line. The family is the universal
$(2,2)$ hypersurface, a Fano fourfold of Picard rank two.  

The variety $Y$ gives rise to other interesting families of 
fourfolds failing stable rationality: double covers \cite{HPT-double}, and conic bundles over $\bP^3$ \cite{auel-co}.      
In this note, we exhibit another natural family
of smooth complex projective fourfolds $X$
 with rational and irrational fibers: Fano fourfolds   
of Picard rank one, obtained as intersections of three quadrics in $\bP^7$.

\begin{theo} 
\label{theo:main}
Let $B\subset \Gr(3,\Gamma(\cO_{\bP^7}(2)))$ be the open subset of the Hilbert
scheme parametrizing smooth complete intersections 
of three quadrics in $\bP^7$ and  
\begin{equation}
\label{eqn:fam}
\phi: \mathcal X \ra B
\end{equation}
the corresponding universal family. 
\begin{enumerate}
\item For very general $b\in B$ the fiber $\mathcal X_b$ is not stably rational.
\item The set of $b\in B$ such that $\mathcal X_b$ is rational is dense in $B$ for the Euclidean topology.  
\end{enumerate}
\end{theo}

\

\noindent {\bf Acknowledgments:} The first author was partially supported by NSF grant 1551514, the second 
author by NSF grant 1601680, and the third by NSF grant 1601912.
We would like to thank Fran\c{c}ois Charles for helpful conversations.

\section{Strategy}
\label{sect:strategy}

We follow the approach in \cite{HPT}. In this section, we
recall the main steps in the proof; details are provided in Section ~\ref{sect:comp}.

\subsection{Fibers that are not stably rational.\\} 
\label{sect:stab-rat}

Recall that a projective variety $X$ over a field $k$ is {\em universally 
$\CH_0$-trivial} if for all field extensions $k'/k$ 
the natural degree homomorphism from the Chow group of zero-cycles
$$
\CH_0(X_{k'})\ra \bZ
$$
is an isomorphism. A projective morphism 
$$
\beta: \tilde{X}\ra X
$$ 
of $k$-varieties 
is {\em universally $\CH_0$-trivial} if for all extensions
$k'/k$ the push-forward homomorphism
$$
\beta_* : \CH_0(\tilde{X}_{k'})\ra \CH_0(X_{k'})
$$
is an isomorphism.

In this paper, we apply the
specialization method of Voisin in the following form.
\begin{theo}
\label{thm:main-help}
\cite[Theorem 2.1]{Voisin}, \cite[Theorem 2.3]{ct-pirutka}
Let 
$$
\phi: \mathcal X\ra B
$$ 
be a flat projective morphism  of complex varieties with smooth generic fiber. 
Assume that there exists a point $b\in B$ such that the fiber 
$$
X:=\phi^{-1}(b)
$$ 
satisfies the following conditions:
\begin{itemize}
\item[(R)] $X$ admits a desingularization 
$$
\beta: \tilde{X}\ra X
$$
such that the morphism $\beta$ is universally $\CH_0$-trivial;
\item[(O)] the variety $\tilde X$ is not universally $\CH_0$-trivial.  
\end{itemize}
Then a very general fiber of $\phi$ is not universally $\CH_0$-trivial; 
in particular, it is not stably rational.
\end{theo}

Condition (O) holds, for instance, if the unramified cohomology
group  $H^2_{nr}(\mathbb C(X)/\bC, \bZ/2)$ is nontrivial.
By \cite[Proposition 1.8]{ct-pirutka} and \cite[Lemma 2.4]{ct-pir-cyclic} condition (R) is satisfied if 
for every scheme point $x$ of $X$, 
the fiber $\beta^{-1}(x)$, considered as a variety over the residue field $\kappa(x)$,  
could be written as $\beta^{-1}(x)=\cup_i X_i$, where
 each component $X_i$ is smooth, geometrically irreducible and $\kappa(x)$-rational and each intersection $X_i\cap X_j$ is either empty or has a zero-cycle of degree 1.

In \cite[Propositions 11, 12]{HPT}, we constructed a hypersurface 
$Y\subset \bP^2\times \bP^3$ of bidegree $(2,2)$, satisfying the
obstruction condition (O) and the resolution condition (R) as above (see \eqref{eqn:Y}). 
The first projection $Y\to \bP^2$ endows $Y$ with a structure of a quadric surface 
bundle with  discriminant curve of degree $8$.
As explained in \cite[Exemple 1.4.4]{Beau77}, smooth intersections of three quadrics in $\bP^7$ 
are also birational to quadric surface bundles over $\bP^2$, 
with discriminant curve of degree $8$ (see Proposition \ref{prop:bundle2} below).
These two families, hypersurfaces of bidegree $(2,2)$ in $\bP^2\times \bP^3$ 
and intersections of three quadrics in $\bP^7$, are genuinely different;
see Section~\ref{sect:differ} for a precise statement.
Both specialize (birationally) to the same reference fourfold: 
in Proposition~\ref{prop:special} we provide an explicit example of a (singular) intersection of three quadrics $X\subset \bP^7$  
such that $X$ is birational to the variety $Y$ above. We deduce  Theorem~\ref{theo:main}, Part (1), from Theorem~\ref{thm:main-help} 
at the end of Section~\ref{secnsr}.

\subsection{  \bf One rational fiber.\\}

Let $\phi: \mathcal X \ra B$ be the family \eqref{eqn:fam}. 
By Proposition \ref{prop:bundle2}, for any $b\in B$, the fiber $\mathcal X_b$ is birational to a 
quadric bundle over $\bP^2$. 
In Section \ref{oneratf} (Proposition \ref{prop:section}), 
we provide an explicit example of a fiber $\mathcal X_b$, birational to a quadric bundle with a rational section. 
In particular, the fourfold $\mathcal X_b$ is rational.

\subsection{  \bf Density of rational fibers.\\}
Let $X\subset \bP^7$ be a smooth intersection of three quadrics. 
As in the previous step, in order to establish that $X$ is rational, 
it suffices to exhibit a  quadric surface bundle $\pi:Q\ra \bP^2$ such that $Q$ is birational to $X$ 
and such that $\pi$ admits a rational section. By Springer's theorem, it suffices to show 
that $\pi$ has a rational multisection of odd degree. 
For quadric bundles this can 
be formulated as a Hodge-theoretic condition:
\begin{prop}
\label{prop:bundle}
\cite[Corollaire 8.2]{CTV}
Let $Q$ be a smooth projective complex algebraic variety, admitting a dominant morphism $\pi: Q\to\mathbb P^2$, 
with generic fiber a quadric of dimension at least $1$. 
Then the integral Hodge conjecture holds for classes of degree $(2,2)$ on $Q$.
\end{prop}

Thus, in order to show that $X$ is rational, it suffices to provide a $(2,2)$-Hodge class intersecting the class of a 
fiber of $\pi$ in odd degree. 
We achieve this by studying the infinitesimal period map. This technique is explained in \cite[5.3.4]{voisin-book}.

The Hodge diamond of $X$ is of the following form:
{\tiny \[
\begin{array}{ccccccccc}
&&&&1&&&& \\[2mm]
&&&0&&0&&& \\[2mm]
\ \ \ \ &&0&&1&&0&& \ \ \ \ \\[2mm]
&0&&0&&0&&0& \\[2mm]
0& &3&&38 &&3&&0\\[2mm]
&0&&0&&0&&0& \\[2mm]
&&0&&1&&0&& \\[2mm]
&&&0&&0&& &\\[2mm]
&&&&1&&&& \\[2mm]
\end{array}
\]}
In particular, the degree $4$ cohomology is essentially of weight $2$. 
We can then apply the following criterion to the family $\mathcal X\to B$ of Theorem~\ref{theo:main}  (cf. \cite[5.3.4]{voisin-book}):
\begin{prop}
Suppose there exists a $b_0\in B$ and $\gamma \in  H^{2,2}(\mathcal X_{b_0})$ 
such that the infinitesimal period map
\begin{equation}\label{period}
\bar{\nabla}:T_{B,b_0} \ra Hom(H^{2,2}(\mathcal X_{b_0}), H^{1,3}(\mathcal X_{b_0})),
\end{equation}
evaluated at $\gamma$, gives a surjective map
\begin{equation}\label{periodev}
\bar{\nabla}(\gamma):T_{B,b_0} \ra H^{1,3}(\mathcal X_{b_0}).
\end{equation}
 Then for any $b\in B$ and any Euclidean neighborhood
$b\in B' \subset B$, the image of the natural map (composition of
inclusion with local trivialization):
\begin{equation}\label{trivialization}
\cH^{2,2}_{\bR} \ra H^4(\mathcal X_b,\bR)
\end{equation}
contains an open subset $V_b \subset H^4(\mathcal X_b,\bR)$.
Here $\cH^{2,2}_{\bR}$ is a vector bundle over $B'$ with fiber over $u$
equal to the real classes of type $(2,2)$ in $H^4(\cX_u)$.
\end{prop}

In order to check the infinitesimal criterion we use an explicit description of the period map:
\begin{prop}
\cite[Corollary 2.5, Proposition 2.6]{Ter}
Let $X\subset \bP^7$ be a smooth complete intersection of three quadrics, defined by  equations 
$$
Q_i(x_0,\ldots, x_7)=0, \quad 
i=0,1,2$$ and let 
$$
F=\mu_0Q_0+\mu_1Q_1+\mu_2Q_2\in \bC[\mu_0, \mu_1,\mu_2, x_0,\ldots, x_7].
$$ 
Let $I\subset \bC[\mu_0, \mu_1,\mu_2, x_0,\ldots, x_7]$ 
be the ideal generated by 
$$
\partial F/\partial \mu_i,\quad i=0,1,2\mbox{ and }\partial F/\partial x_i,\quad  i=0,\ldots, 7.
$$ 
Put
$$
R=\bC[\mu_0, \mu_1,\mu_2, x_0,\ldots, x_7]/I
$$ 
and let $R_{(a,b)}$ be the space of homogeneous elements of degree $(a,b)$ in $R$, with respect to the  grading $(\mu,x)$. Then there is an isomorphism
$$H^{4-q,q}_{\mathrm{prim}}(X)\simeq R_{(q,2q-2)}$$
 and the period map (\ref{period})  is identified with the multiplication homomorphism
\begin{equation}\label{periodexplicit}
R_{(1,2)}\otimes R_{(2,2)}\to R_{(3,4)}. 
\end{equation}
\end{prop}
 
Recall that the primitive cohomology $H^{p,q}_{\mathrm{prim}}$ is the cokernel of the natural map $H^{p,q}(\bP^7)\to H^{p,q}( X)$. 

In Section \ref{exdensity}, we provide an explicit example $X=\mathcal X_{b_0}$ such that the period map \ref{periodexplicit} is surjective  (Proposition \ref{prop:surj}). 
Theorem~\ref{theo:main}, Part (2), then follows. In fact, by Proposition \ref{oneratf}, 
there exists a smooth intersection of three quadrics birational to a quadric bundle with a rational section. 
Similarly to \cite[Proposition 14]{HPT} the density of rational fibers follows from the infinitesimal criterion  that we verify in Proposition \ref{prop:surj}.

\section{Computations} 
\label{sect:comp}

We work over the complex numbers.  We first recall the construction of Beauville \cite[Exemple 1.4.4]{Beau77}: 

\begin{prop}
\label{prop:bundle2}
Let $X\subset \bP^7$ be a smooth complete intersection of three quadrics. 
Then $X$ is birational to a quadric bundle  over $\bP^2$, with discriminant curve  of degree $8$. 

Concretely, let $\ell\subset X$ be a line and 
$G_{\ell}\simeq\mathbb P^5$ the space of $2$-planes $\Pi\subset \mathbb P^7$ containing $\ell$. 
Then $X$ is birational to a quadric surface bundle 
$$
\pi: Q\to \mathbb P^2,
$$ 
where $Q\subset \mathbb P^2\times G_{\ell}$ is given by
\begin{equation}\label{geobundle}
Q=\left\{([\lambda_0:\lambda_1:\lambda_2], \Pi) |\quad\{ \lambda_0Q_0+\lambda_1Q_1+\lambda_2Q_2 =0\}\supset \Pi\right\}.
\end{equation}
\end{prop}

More explicitly, assume that the line is given by equations
$$
\ell: x_2=x_3=\ldots=x_7=0
$$ 
and write, for $i=0,1,2$,  
$$
Q_i=x_0L_i(x_2,x_3,\ldots, x_7)+x_1M_i(x_2,x_3,\ldots, x_7)+q_i(x_2,x_3,\ldots, x_7),
$$
where $L_i$ and $M_i$ are linear forms and $q_i$ is quadratic. Any $2$-plane $\Pi\subset\mathbb P^7$ 
containing $\ell$ intersects the $5$-plane $x_0=x_1=0$ in a unique point $[0:0:x_2:\cdots : x_7]$.  
This allows us to identify the space of $2$-planes $\Pi\subset \mathbb P^7$ containing $\ell$ with $\mathbb P^5$. 
Then the quadric bundle (\ref{geobundle}) is defined in $\bP^2\times \bP^5$  by the equations 
\begin{multline}\label{bundlexplicit}
\sum_{i=0}^2 \lambda_iL_i(x_2,x_3,\ldots, x_7)=\sum_{i=0}^2 \lambda_iM_i(x_2,x_3,\ldots, x_7)=\\=\sum_{i=0}^2 \lambda_iq_i(x_2,x_3,\ldots, x_7)=0. 
\end{multline}

\subsection{Fibers that are not stably rational}\label{secnsr}

Let  $X\subset \bP^7$ be the intersection of three quadrics\\
\begin{multline}\label{Xspecial}
Q_0: \,-x_0x_5+x_3^2+x_4x_6-2x_5^2=0;\\
Q_1: \, x_0x_5+x_1x_4+x_2^2-2x_5^2=0;\\
Q_2: x_0x_7-x_1x_6+x_5^2+x_7^2=0.
\end{multline}

Note that $X$ contains a line $\ell: x_2=\ldots=x_7=0$. Using equations (\ref{bundlexplicit}), we obtain that $X$ is 
birational to a quadric bundle $Q\to \bP^2$, defined in $\bP^2\times \bP^5$ as an intersection of two forms of bidegree $(1,1)$ and one
form of bidegree $(1,2)$: 
\begin{multline}\label{Qspecial}
(\lambda_0-\lambda_1)x_5=\lambda_2x_7, \quad \;
\lambda_1x_4=\lambda_2x_6\\
\lambda_1x_2^2+\lambda_0x_3^2+\lambda_0x_4x_6+(\lambda_2-2\lambda_0-2\lambda_1)x_5^2+\lambda_2x_7^2=0.
\end{multline}
In the open set $\lambda_2\neq 0$ we can define $X$ 
by a single equation
$$
\lambda_1x_2^2+\lambda_0x_3^2+\frac{\lambda_0\lambda_1}{\lambda_2}x_4^2+(\frac{(\lambda_0-\lambda_1)^2}{\lambda_2}+\lambda_2-2\lambda_0-2\lambda_1)x_5^2=0,
$$
hence, $X$ is birational to a hypersurface $Y\subset \bP^2\times \bP^3$ of bidegree $(2,2)$ defined by
\begin{equation}\label{ourExample}
\lambda_1\lambda_2x_2^2+\lambda_0\lambda_2x_3^2+\lambda_0\lambda_1x_4^2+F(\lambda_0,\lambda_1,\lambda_2)x_5^2=0,
\end{equation}
where $F(\lambda_0,\lambda_1,\lambda_2)=\lambda_0^2+\lambda_1^2+\lambda_2^2-2\lambda_0\lambda_1-2\lambda_0\lambda_2-2\lambda_1\lambda_2$.

\noindent This is precisely the hypersurface we considered in  \cite[Propositions 11, 12]{HPT}.

\begin{prop}
\label{prop:special}
Let $Q\subset \bP^2\times \bP^5$ be defined by the equations (\ref{Qspecial}) and let $Y\subset \bP^2\times \bP^3$ be the hypersurface given by the equation $(\ref{ourExample})$. Then
the birational map
\begin{multline}\label{map}
\varphi:Y\dashrightarrow Q,\\ (\lambda_0:\lambda_1:\lambda_2, x_2:\ldots: x_5)\mapsto\\ (\lambda_0:\lambda_1:\lambda_2, \lambda_2 x_2: \lambda_2x_3:\lambda_2x_4:\lambda_2 x_5:\lambda_1x_4:(\lambda_0-\lambda_1)x_5)
\end{multline}
extends to the following diagram
$$
\xymatrix{&\tilde Y\ar_{\psi}[ld]\ar^{\tilde \varphi}[rd] &\\
Y\ar@{-->}[rr]&&Q
}
$$
 where the morphisms $\psi:\tilde Y\to Y$ and $\tilde \varphi:\tilde Y\to Q$ are birational and universally $\CH_0$-trivial.
\end{prop}
\begin{proof} First note that $\varphi$ is indeed a birational map between $Y$ and $Q$.
The locus $Y^{nd}\subset Y$ where the map $\varphi$ is not defined is a union of three components\\
$Y_1: \lambda_2=0, x_4=x_5=0;$\\
$Y_2: \lambda_1=\lambda_2=0, x_5=0$;\\
$Y_3: \lambda_0-\lambda_1=0, \lambda_2=0, x_4=0$.\\
Note that $Y_1$ is isomorphic to a product  $\bP^1_{\lambda_0:\lambda_1}\times \bP^1_{x_2:x_3}$, and similarly $Y_2$  is isomorphic to a projective plane $\bP^2_{x_2:x_3:x_4}$ with homogeneous coordinates $[x_2:x_3:x_4]$ and $Y_3\simeq \bP^2_{x_2:x_3:x_5}$.

We construct $\tilde Y$ by successive blowups  of $Y_1$,  the proper transform of $Y_2$ and the proper transform of $Y_3$. After each blowup we verify:
\begin{itemize}
\item the indeterminacy locus of $\varphi$ on the blowup;
\item the universal $\CH_0$-triviality of fibers of the extension
of $\varphi$ to the blowup 
and of the blowup map. In each case we obtain that the corresponding fiber is either reduced to a point or projective (or affine, if we compute on open charts) spaces. We provide details  for the first computations and the expressions in the coordinates for the remaining charts.\\
\end{itemize}

{\it Blowup of $Y_1$.} We have three charts:
\begin{enumerate}
\item $U_1: x_4=\lambda_2u_4$, $x_5=\lambda_2u_5$,  the exceptional divisor is given by $\lambda_2=0$.  Since we blow up the locus  $\lambda_2=0, x_4=x_5=0$, we consider one of the charts $\lambda_0\neq 0$ or $\lambda_1\neq 0$ of $\bP^2$ and one of the charts $x_2\neq 0$ or $x_3\neq 0$ of $\bP^3$. \\
We extend $\varphi$ to a birational map $\varphi_1:U_1\dashrightarrow Q$,
$$(\lambda_0,\lambda_1, \lambda_2, x_2, x_3, u_4, u_5)\mapsto (\lambda_0, \lambda_1,\lambda_2, x_2, x_3, \lambda_2u_4, \lambda_2u_5, \lambda_1u_4, (\lambda_0-\lambda_1)u_5).$$ Since one of coordinates $\lambda_0,\lambda_1$ is nonzero, and one of coordinates $x_2, x_3$ is nonzero, we have that  $\varphi_1$ is  well-defined. The image of $\varphi_1$ is contained in the closure of the image of $\varphi$, hence it is contained in $Q$, so that we obtain a map  $\varphi_1: U_1\to Q$.

The image of the exceptional divisor is the set of points 
$$
E_1=(\lambda_0, \lambda_1, 0, x_2, x_3, 0,0, \lambda_1u_4, (\lambda_0-\lambda_1)u_5).
$$ 
Then for any field $k'/\bC$ and for any point $P\in E_1(k')$ the fiber $\varphi_1^{-1}(P)$ is either a point or a line (if $\lambda_1=0$ or $\lambda_0-\lambda_1=0$), which ensures the universal $\CH_0$-triviality of the map $\varphi_1$ on this chart. 

The equation defining $U_1$ is
$$\lambda_1x_2^2+\lambda_0x_3^2+\lambda_0\lambda_1\lambda_2u_4^2+F(\lambda_0, \lambda_1, \lambda_2)\lambda_2u_5^2=0.$$
Let $\psi_1: U_1\to Y$ be the blowup map. Then, the image $I_1$ of the exceptional divisor is given by the conditions 
$$
\lambda_2=0,\quad \lambda_1x_2^2+\lambda_0 x_3^2=0.
$$ 
The latter condition defines a point since the coordinates $\lambda_0:\lambda_1$ and $x_2:x_3$ are homogeneous.
Then for any field $k'/\bC$ and for any point $P\in I_1(k')$ the fiber $\psi_1^{-1}(P)$ is  a plane with coordinates $u_4$ and $u_5$, which ensures the universal $\CH_0$-triviality of the map $\psi_1$ on this chart.

\item $U_2:$
\begin{itemize}
\item  change of variables: $$ \lambda_2=x_4\lambda_2', x_5=x_4u_5;$$
\item equation defining the blowup: $$\lambda_1\lambda_2'x_2^2+\lambda_0\lambda_2^2x_3^2+\lambda_0\lambda_1x_4+F(\lambda_0, \lambda_1, \lambda_2'x_4)x_4u_5^2=0.$$
\item  exceptional divisor: $$x_4=0,  \lambda_1\lambda_2'x_2^2+\lambda_0\lambda_2^2x_3^2=0.$$
\item extension of $\varphi$ is given by: $$(\lambda_0, \lambda_1,\lambda_2'x_4, \lambda_2'x_2, \lambda_2'x_3, \lambda_2'x_4, \lambda_2'x_4u_5, \lambda_1, (\lambda_0-\lambda_1)u_5).$$
\item domain, where the extension is not defined is the proper transform $Y_2'$ of $Y_2$: $$\lambda_1=\lambda_2'=0, u_5=0.$$
\item the image of the exceptional divisor: $$(\lambda_0, \lambda_1, 0, \lambda_2'x_2, \lambda_2'x_3, 0, 0, \lambda_1, (\lambda_0-\lambda_1)u_5).$$
\end{itemize}

\item $U_3: $
\begin{itemize}
\item  change of variables:  $$\lambda_2=x_5\lambda_2', x_4=x_5u_4;$$
\item equation defining the blowup: $$\lambda_1\lambda_2'x_2^2+\lambda_0\lambda_2'x_3^2+\lambda_0\lambda_1x_5u_4^2+F(\lambda_0, \lambda_1, \lambda_2'x_5)x_5=0;$$
\item  exceptional divisor: $$x_5=0, \lambda_1\lambda_2'x_2^2+\lambda_0\lambda_2^2x_3^2=0;$$
\item extension of $\varphi$ is given by:
$$(\lambda_0, \lambda_1,\lambda_2'x_5, \lambda_2'x_2, \lambda_2'x_3, \lambda_2'x_5u_4, \lambda_2'x_5, \lambda_1u_4, \lambda_0-\lambda_1)$$
\item domain, where the extension is not defined is the proper transform $Y_3'$ of $Y_3$:
$$\lambda_0-\lambda_1=\lambda_2'=0, u_4=0.$$
\item the image of the exceptional divisor: $$(\lambda_0, \lambda_1, 0, \lambda_2'x_2, \lambda_2'x_3, 0, 0, \lambda_1u_4, \lambda_0-\lambda_1).$$
\end{itemize}

\end{enumerate}

{\it Blowup of the proper transforms $Y_2'$ and $Y_3'$}

Note that $Y_2$ and $Y_3$, and hence  their proper transforms, do not intersect. Hence we can use charts $U_2$ and $U_3$ independently for their blowups.

\begin{enumerate}
\item On the chart $U_2$:
\begin{enumerate}
\item \begin{itemize}
\item  change of variables: $$\lambda_1=\lambda_2'\lambda_1', u_5=\lambda_2'v_5$$
\item exceptional divisor:  $$\lambda_2'=0,\lambda_0x_3^2+\lambda_0\lambda_1'x_4;$$
\item extension of $\varphi$ is everywhere defined:
$$(\lambda_0,\lambda_1'\lambda_2', x_4\lambda_2', x_2, x_3, x_4, \lambda_2'x_4v_5, \lambda_1', (\lambda_0-\lambda_1'\lambda_2')v_5);$$
\item the image of the exceptional divisor: $$(1,0,0, x_2, x_3, x_4, x_4v_5, \lambda_1', v_5).$$
\end{itemize}

\item
\begin{itemize}
\item  change of variables:  $$\lambda_2'=\lambda_1\lambda_2'', u_5=\lambda_1v_5;$$
\item exceptional divisor: $$\lambda_1=0, \lambda_0\lambda_2''x_3^2+\lambda_0x_4=0;$$
\item extension of $\varphi$ is everywhere defined:
$$(\lambda_0:\lambda_1:\lambda_1\lambda_2'',\lambda_2'' x_2, \lambda_2''x_3, \lambda_2''x_4, \lambda_1\lambda_2''x_4v_5,1, v_5(\lambda_0-\lambda_1));$$
\item the image of the exceptional divisor: $$(1,0,0, \lambda_2''x_2, \lambda_2''x_3, \lambda_2''x_4, 0, 1, v_5).$$
\end{itemize}

\item
\begin{itemize}
\item  change of variables: $$\lambda_2'=u_5\lambda_2'', \lambda_1=u_5\lambda_1'';$$
\item exceptional divisor: $$u_5=0, \lambda_0\lambda_2''x_3^2+\lambda_0\lambda_1''x_4=0;$$
\item extension of $\varphi$ is everywhere defined:
$$(\lambda_0,\lambda_1''u_5,\lambda_2''u_5,\lambda_2'' x_2, \lambda_2''x_3, \lambda_2''x_4, \lambda_2''x_4u_5, \lambda_1'',\lambda_0-\lambda_1''u_5);$$ 
\item the image of the exceptional divisor: $$(1,0,0, \lambda_2''x_2, \lambda_2''x_3, \lambda_2''x_4, 0, \lambda_1'', 1).$$
\end{itemize}

\end{enumerate}

\item On the chart $U_3$:
\begin{enumerate}
\item \begin{itemize}
\item  change of variables: $$\lambda_0-\lambda_1=\lambda_2'\lambda_0', u_4=\lambda_2'v_4;$$ 
\item exceptional divisor: $$\lambda_2'=0, \lambda_1x_2^2+\lambda_1x_3^2-4\lambda_1x_5=0;$$
\item extension of $\varphi$  is everywhere defined:
$$(\lambda_1+\lambda_2'\lambda_0', \lambda_1, \lambda_2'x_5, x_2, x_3, \lambda_2'x_5v_4, x_5, \lambda_1v_4, \lambda_0');$$
\item the image of the exceptional divisor: $$(\lambda_1, \lambda_1, 0, x_2, x_3, 0,x_5, \lambda_1v_4, \lambda_0').$$
\end{itemize}

\item \begin{itemize}
\item  change of variables: $$\lambda_2'=(\lambda_0-\lambda_1)\lambda_2'', u_4=(\lambda_0-\lambda_1)v_4;$$
\item exceptional divisor: $$(\lambda_0-\lambda_1)=0, \lambda_1\lambda_2''x_2^2+\lambda_1\lambda_2''x_3^2-4\lambda_1\lambda_2''x_5=0;$$
\item extension of $\varphi$  is everywhere defined:
$$(\lambda_0, \lambda_1, \lambda_2''(\lambda_0-\lambda_1)x_5, \lambda_2''x_2, \lambda_2''x_3, (\lambda_0-\lambda_1)\lambda_2''x_5v_4,\lambda_2''x_5, \lambda_1v_4, 1 );$$
\item the image of the exceptional divisor: $$(\lambda_1, \lambda_1, 0, \lambda_2''x_2, \lambda_2''x_3, 0, \lambda_2''x_5, \lambda_1v_4, 1).$$
\end{itemize}

\item \begin{itemize}
\item  change of variables: $$\lambda_2'=u_4\lambda_2'', \lambda_0-\lambda_1=u_4\lambda_0';$$
\item exceptional divisor: $$u_4=0, \lambda_1\lambda_2''x_2^2+\lambda_1\lambda_2''x_3^2-4\lambda_1\lambda_2''x_5=0;$$
\item extension of $\varphi$  is everywhere defined:
$$(\lambda_1+u_4\lambda_0', \lambda_1, \lambda_2''u_4x_5, \lambda_2''x_2, \lambda_2''x_3, \lambda_2''x_5u_4, \lambda_2''x_5, \lambda_1, \lambda_0');$$
\item the image of the exceptional divisor:  $$(\lambda_1, \lambda_1, 0, \lambda_2''x_2, \lambda_2''x_3, 0, \lambda_2''x_5, \lambda_1, \lambda_0').$$
\end{itemize}
\end{enumerate}

\end{enumerate}

\end{proof}

\begin{coro}\label{coro:special2}
Let $Q\subset \bP^2\times \bP^5$ be defined by the equations (\ref{Qspecial}). Then $Q$ admits 
a resolution of singularities $\beta:\tilde Q\to Q$ such that  
\begin{itemize}
\item[(i)] the variety $\tilde Q$ is not universally $\CH_0$-trivial; 
\item[(ii)] the map $\beta$ is a universally $\CH_0$-trivial morphism.
\end{itemize}
\end{coro}
\begin{proof}
We use Proposition~\ref{prop:special}: $Q$ is birational to a variety $Y$ with $H^2_{nr}(\bC(Y)/\bC,\bZ/2)\neq 0$ by  \cite[Proposition 11]{HPT}. In particular, property (i) holds for any resolution $\tilde Q$ of $Q$. 

In  \cite[Proposition 12]{HPT} we constructed a resolution of singularities $f:Z\to Y$ such that $f$ is 
a universally $\CH_0$-trivial morphism. Then there is birational map $\tilde f:\tilde Z\to Z$ with $\tilde Z$ smooth, 
such that the rational map $Z\dashrightarrow \tilde Y$ extends to a map $\tilde Z\to \tilde Y$:
$$
\xymatrix{
\tilde Z\ar_{\tilde f}[d]\ar[dr]& &\\
Z\ar_{f}[d]\ar@{-->}[r]&\tilde Y\ar_{\psi}[ld]\ar^{\tilde \varphi}[rd] &\\
Y\ar@{-->}[rr]&&Q
}
$$

Note that the map $\tilde f$ is universally $\CH_0$-trivial: by weak factorization, $\tilde f$ factors through blow-ups and blow-downs at smooth centers, each of these maps is universally $\CH_0$-trivial.  Hence, in the diagram above, the maps $\tilde f,f,\psi,\tilde \varphi$ are universally $\CH_0$-trivial. We deduce from the diagram that the composite map $\tilde Z\to Q$ is also universally $\CH_0$-trivial, which shows (ii).
\end{proof}

\noindent {\it Proof of Theorem~\ref{theo:main}, Part (1):}\\
From Theorem~\ref{thm:main-help} and Corollary~\ref{coro:special2}  we deduce that a very general quadric bundle defined by equations (\ref{bundlexplicit}) is not universally $\CH_0$-trivial. In particular, there exists a smooth intersection of three quadrics $X$ birational to a smooth quadric bundle $Q$ defined by an equation of type (\ref{bundlexplicit}), such that $Q$ is  not universally $\CH_0$-trivial. Since universal $\CH_0$-triviality is a birational invariant of smooth projective varieties, we deduce that $X$ is not universally $\CH_0$-trivial. Then Theorem~\ref{theo:main}, Part (1),  follows directly from Theorem~\ref{thm:main-help}, applied to the universal family $\phi:\mathcal X \ra B$ of 
smooth complete intersections of three quadrics in $\bP^7$. 
\qed

\subsection{One rational fiber}
\label{oneratf}

Consider the quadrics

\begin{multline*}
Q_0: \quad x_0(x_3+x_5+2x_6+3x_7)+x_1(-x_5+5x_6+2x_7)-\\-x_2x_3-x_2x_4+x_2x_5+x_3^2-x_4x_6+x_5^2+x_6^2+x_7^2=0;\\
Q_1: \quad x_0(-x_2+3x_5+7x_6+11x_7)+x_1(x_4+9x_5+4x_6+x_7)+\\+x_2^2-x_2x_3+2x_3x_6+x_4^2+3x_4x_7+2x_5^2+3x_6^2+5x_7^2=0;\\
Q_2: \quad x_0(11x_5+13x_6+8x_7)+x_1(-x_3+6x_5+7x_6+3x_7)+\\+x_2^2+5x_2x_7-x_3x_4+9x_3x_5+13x_5^2+4x_6^2+11x_7^2=0.
\end{multline*}

\begin{prop}
\label{prop:section} 
Let $X\subset \bP^7$ be the intersection 
$$
Q_0=Q_1=Q_2=0
$$
Then $X$ is smooth and rational. 
\end{prop}

\begin{proof}
A Magma \cite{Magma} computation shows that $X$ is smooth. Furthermore, 
$X$ contains a line 
$$
\ell: x_2=\ldots=x_7=0.
$$ 
As in Proposition \ref{prop:bundle2}, considering the space
$G_{\ell}\simeq\mathbb P^5$ of $2$-planes $\Pi\subset\mathbb P^7$ containing $\ell$, 
we find that $X$ is birational to a fibration in quadrics $Q\to \mathbb P^2$, 
where $Q\subset \mathbb P^2\times G_{\ell}$, 
$$
Q=\{([\lambda_0:\lambda_1:\lambda_2], \Pi)|\quad \{ \lambda_0Q_0+\lambda_1Q_1+\lambda_2Q_2=0\} \supset \Pi\}.
$$
The first projection $Q\to \mathbb P^2$ admits a rational section: 
the plane containing $\ell$ and the point 
$[0:0:\lambda_0:\lambda_1:\lambda_2:0:0:0]$ is contained in $\{ \lambda_0Q_0+\lambda_1Q_1+\lambda_2Q_2=0\}$.
Indeed, by (\ref{bundlexplicit}), we have that $Q\subset \bP^2\times \bP^5$ is defined by the equations:
\begin{multline*}
\lambda_0(x_3+x_5+2x_6+3x_7)+\lambda_1(-x_2+3x_5+7x_6+11x_7)+\lambda_2(11x_5+13x_6+8x_7)=0\\
\lambda_0(-x_5+5x_6+2x_7)+\lambda_1(x_4+9x_5+4x_6+x_7)+\lambda_2(-x_3+6x_5+7x_6+3x_7)=0\\
\lambda_0(-x_2x_3-x_2x_4+x_2x_5+x_3^2-x_4x_6+x_5^2+x_6^2+x_7^2)+\lambda_1(x_2^2-x_2x_3+2x_3x_6+x_4^2+\\+3x_4x_7+2x_5^2+3x_6^2+5x_7^2)+\lambda_2(x_2^2+5x_2x_7-x_3x_4+9x_3x_5+13x_5^2+4x_6^2+11x_7^2)=0
\end{multline*}
and, substituting 
$$
[x_2:x_3:\ldots:x_7]=[0:0:\lambda_0:\lambda_1:\lambda_2:0:0:0],
$$ 
we obtain
\begin{multline*} \lambda_0\lambda_1-\lambda_0\lambda_1=0,\quad  \lambda_1\lambda_2-\lambda_1\lambda_2=0,\\
\lambda_0(-\lambda_0\lambda_1-\lambda_0\lambda_2+\lambda_1^2)+\lambda_1(\lambda_0^2+\lambda_2^2-\lambda_0\lambda_1)+\lambda_2(\lambda_0^2-\lambda_1\lambda_2)=0.
\end{multline*}
\end{proof}

\subsection{Density of rational fibers}\label{exdensity}

Using the notation of Section~\ref{oneratf}, 
consider quadrics 
$$
\begin{array}{rcl}
Q_0'& := & Q_0+x_0^2+x_5^2\\
Q_1'& := & Q_1\\
Q_2'& := & Q_2+x_1^2+x_3^2
\end{array}
$$

\begin{prop}
\label{prop:surj}
Let $X'\subset \bP^7$ be the intersection 
$$
Q_0'=Q_1'=Q_2'=0.
$$
Then $X'$ is smooth and there exists a $\gamma \in  H^{2,2}(X')$ such that the period map (\ref{periodev}) is surjective.
\end{prop}

\begin{proof}
A Magma computation shows that $X'$ is smooth. 
In order to compute the period map we use  expression (\ref{periodexplicit}). We used Macaulay2 \cite{M2} to verify that the following monomials 
$$
\{\mu_0\mu_2^2x_7^4, \;\mu_1\mu_2^2x_7^4,\; \mu_2^3x_7^4 \}
$$
form a basis of the graded part $R_{(3,4)}\simeq H^{1,3}(X')$. 
In particular $\gamma=\mu_2^2x_7^2$ works.
\end{proof}

\section{Differentiating quadric bundles} \label{sect:differ}

The goal of this section is to show that the quadric bundles
arising from complete intersection of three quadrics in $\bP^7$ do in fact
differ from the $(2,2)$ hypersurfaces in $\bP^2 \times \bP^3$ considered
in \cite{HPT}. Note however that both families specialize to the {\em same}
reference variety (\ref{eqn:Y}).

Let $\pi:Q \ra \bP^2$ be a quadric surface bundle with smooth
degeneracy curve $D\subset \bP^2$, i.e., $Q$ is a smooth complex projective
fourfold, $\pi$ is a flat morphism with smooth ($\simeq \bF_0$) fibers over
$\bP^2 \setminus D$, and quadric cones ($\simeq \bP(1,1,2)$)
as fibers over $D$. Let $\tau:S\ra \bP^2$ denote the associated double cover,
simply branched along $D$. We may interpret $S$ as the Stein factorization
of the relative variety of lines
$$F_1(Q/\bP^2) \ra S \ra \bP^2;$$
as such, $S$ is equipped with a natural conic bundle structure and
thus a class $\alpha_Q \in H^2(S,\mu_2)$.
We refer the reader to \cite{APS} for
a close analysis of the equivalence between quadric surface bundles and 
Azumaya algebras over double covers. 

We present a cohomological interpretation of this correspondence due to Laszlo
\cite{Las}. 
Let $H^2_0(S,\bZ)$ denote the primitive cohomology of $S$, i.e., the 
kernel of $\tau_*$. It carries the structure of a lattice with respect
to the intersection form, as well as a weight two Hodge structure.
Choose an embedding
$$\begin{array}{ccc}
Q & \hookrightarrow & \bP(E) \\
  & \stackrel{\pi}{\searrow} & \downarrow \\
  &			    & \bP^2
\end{array}
$$
where $E\ra \bP^2$ is a rank four vector bundle. Let 
$H^4_0(Q,\bZ)$ denote kernel of the push forward homomorphism
$$H^4(Q,\bZ) \ra H^6(\bP(E),\bZ).$$ 
This carries the structure of a lattice and a weight four Hodge structure.
Let $H^4_0(Q,\bZ)(1)$ denote its Tate twist, a weight two Hodge structure;
this reverses the sign of the integral quadratic form.

\begin{theo} \cite[Th.~II.3.1]{Las} \label{theo:Las}
There exists an embedding of abelian groups
$$\Phi:H^4_0(Q,\bZ)(1) \hookrightarrow H^2_0(S,\bZ)$$
compatible with the lattice and Hodge structures. 
The image has index two and is characterized as follows:
$$\operatorname{image}(\Phi)=\Lambda_Q:=\{\gamma \in H^2_0(S,\bZ): 
\left(\gamma\,\mathrm{mod}\, 2,\alpha_Q\right)\equiv 0\, \mathrm{mod}\, 2 \}.$$
\end{theo}

Now suppose we have a birational equivalence
$$\begin{array}{ccccc}
  Q_1 & & \stackrel{\sim}{\dashrightarrow } & & Q_2 \\
      & \searrow &   & \swarrow & \\
      &          & \bP^2 &      & 
\end{array}
$$
of quadric bundles over $\bP^2$. It is clear that $Q_1$
and $Q_2$ must have the same degeneracy curve $D\subset \bP^2$
and induced double cover $\tau:S\ra \bP^2$.
Consider the classes $\alpha_{Q_1},\alpha_{Q_2} \in \Br(S)[2]$,
obtained via the canonical surjection $H^2(S,\mu_2) \ra \Br(S)[2]$.
Since $\alpha_{Q_i}$ generates the kernel
of 
$$H^2(\bC(S),\mu_2) \ra H^2(\bC(Q_i),\mu_2)$$
by \cite[p.469]{arason}, we have $\alpha_{Q_1}=\alpha_{Q_2}$.

\begin{prop}
Let $D\subset \bP^2$ be a very general octic plane curve,
$Q_1,Q_2 \ra \bP^2$ quadric surface bundles with degeneracy curve
$D$, where $Q_1\subset \bP^2 \times \bP^3$ is a $(2,2)$
hypersurface and $Q_2 \subset \bP^2 \times \bP^5$ is a 
complete intersection of
hypersurfaces of bidegrees $(1,1)$, $(1,1)$, $(1,2)$. Then $Q_1$ 
and $Q_2$ are not birational over $\bP^2$.
\end{prop}
The precise condition we require is that $\operatorname{Pic}(S)\simeq \bZ$.
\begin{proof}
For the first example, 
let $h_1$ and $h_2$  
denote the pull-backs of the hyperplane classes from each factor. Then we 
have $[Q_1]=2h_1+2h_2$ and
$$
\begin{array}{c|ccc}
    & 	h_1^2 &	h_1h_2  &	h_2^2 \\
\hline
h_1^2 &	0 & 	0 &	2 \\
h_1h_2 & 0 &	2 &	2 \\
h_2^2 &	2 &	2 & 0  
\end{array}
$$
For the second example, let $g_1$ and
$g_2$ denote the 
hyperplace classes as above so that 
$$[Q_2] = 4g_1^2g_2 + 5g_1g_2^2 + 2g_2^3.$$
Then we have
$$
\begin{array}{c|ccc}
 & g_1^2 &	g_1g_2 & g_2^2 \\
\hline
g_1^2& 	0 &	0 &	2 \\
g_1g_2 & 0 &	2 &	5 \\
g_2^2 &	2 &	5 &	4
\end{array}
$$
These two lattices are inequivalent over the 2-adics. Indeed, their ranks 
modulo two differ. It follows that the lattices $H^4_0(Q_1,\bZ)$ and
$H^4_0(Q_2,\bZ)$ are also inequivalent, as a nondegenerate lattice
and its orthogonal complement in a unimodular lattice have the 
same discriminant groups up to sign. (The discriminant groups are 
a way of packaging the $p$-adic invariants of a lattice.)

Under our assumption, $\Br(S)[2]=H^2(S,\mu_2)/\left<h\right>$ where
$h$ is the hyperplane class pulled back from $\bP^2$. 
If $Q_1$ and $Q_2$ were birational over $\bP^2$ then 
$$\alpha_{Q_1}=\alpha_{Q_2} \in H^2(S,\mu_2)/\left<h\right>,$$
whence $\Lambda_{Q_1}\simeq \Lambda_{Q_2}$.
This would contradict Theorem~\ref{theo:Las}.
\end{proof}

\begin{rema}
Observe that the common reference variety (\ref{eqn:Y}) admits 
nontrivial $2$-torsion in its unramified cohomology. 
It is intriguing that we
differentiate the smooth members through a $2$-adic computation
of lattices.
\end{rema}

\bibliographystyle{alpha}
\bibliography{3quadrics}
\end{document}